\documentclass{amsart}
\usepackage{amssymb}
\usepackage{color}
\usepackage{graphicx}
\theoremstyle{plain}
\newtheorem{theorem}{Theorem}[section]
\newtheorem{corollary}[theorem]{Corollary}
\newtheorem{definition}[theorem]{Definition}
\newtheorem{proposition}[theorem]{Proposition}

\newtheorem{lemma}[theorem]{Lemma}
\numberwithin{equation}{section}

\makeatletter
\@namedef{subjclassname@2020}{%
  \textup{2020} Mathematics Subject Classification}
\makeatother

\begin{document}

\title[Hankel Determinants for certain Bernoulli and Euler Polynomials]{Orthogonal polynomials and Hankel Determinants for certain Bernoulli and Euler Polynomials}

\author{Karl Dilcher}
\address{Department of Mathematics and Statistics\\
         Dalhousie University\\
         Halifax, Nova Scotia, B3H 4R2, Canada}
\email{dilcher@mathstat.dal.ca}

\author[Lin Jiu]{Lin Jiu\textsuperscript{*}}
\thanks{*corresponding author}
\address{Department of Mathematics and Statistics\\
         Dalhousie University\\
         Halifax, Nova Scotia, B3H 4R2, Canada}
\email{lin.jiu@dal.ca}

\thanks{Research supported in part by the Natural Sciences and Engineering
        Research Council of Canada, Grant \# 145628481}

\date{}

\setcounter{equation}{0}

\begin{abstract}
Using continued fraction expansions of certain polygamma functions as a
main tool, we find orthogonal polynomials with respect to the odd-index 
Bernoulli polynomials $B_{2k+1}(x)$ and the Euler polynomials $E_{2k+\nu}(x)$,
for $\nu=0, 1, 2$. In the process we also determine the corresponding Jacobi
continued fractions (or J-fractions) and Hankel determinants. In all these
cases the Hankel determinants are polynomials in $x$ which factor completely
over the rationals.
\end{abstract}

\maketitle

\section{Introduction}

A {\it Hankel matrix} or {\it persymmetric matrix} is a symmetric matrix which 
has constant entries along its antidiagonals; in other words, it is of the form
\begin{equation}\label{1.0}
\big(c_{i+j}\big)_{0\leq i,j\leq n}
=\begin{pmatrix}
c_{0} & c_{1} & c_{2} & \cdots & c_{n}\\
c_{1} & c_{2} & c_{3} & \cdots & c_{n+1}\\
c_{2} & c_{3} & c_{4} & \cdots & c_{n+2}\\
\vdots & \vdots & \vdots & \ddots & \vdots\\
c_{n} & c_{n+1} & c_{n+2} & \cdots & c_{2n}
\end{pmatrix}.
\end{equation}
A {\it Hankel determinant} is then the determinant of a Hankel matrix. 
Furthermore, given a sequence ${\bf c}=(c_0,c_1,\ldots)$ of numbers or 
polynomials, we define the $n$th Hankel determinant of ${\bf c}$ to be 
\[
H_n({\bf c}) = H_n(c_k) = \det_{0\leq i,j\leq n}\big(c_{i+j}\big).
\]
If we use the second notation, $H_n(c_k)$, it is always assumed that the 
sequence begins with $k=0$; it should be noted that the value of the Hankel
determinant depends on this in an essential way. We shall return to this issue in Section \ref{sec:ShiftedSequences}.

The Hankel determinants of a sequence are closely related to certain orthogonal 
polynomials and continued fractions. It is the purpose of this paper to obtain 
new results, including new Hankel determinants, for certain 
sequences of Bernoulli and Euler polynomials. We recall that the 
{\it Bernoulli numbers} $B_n$
and {\it polynomials} $B_n(x)$ are usually defined by the generating functions
\begin{equation}\label{1.1}
\frac{t}{e^t-1} = \sum_{n=0}^\infty B_n\frac{t^n}{n!}\qquad\hbox{and}\qquad
\frac{te^{xt}}{e^t-1} = \sum_{n=0}^\infty B_n(x)\frac{t^n}{n!}.
\end{equation}
We have $B_0=1$, $B_1=-1/2$, and $B_{2j+1}=0$ for $j\geq 1$; a few further
values are listed in Table~2. The {\it Euler numbers} $E_n$ and 
{\it polynomials} $E_n(x)$ are defined by the generating functions
\begin{equation}\label{1.3}
\frac{2}{e^t+e^{-t}} = \sum_{n=0}^\infty E_n\frac{t^n}{n!}\qquad\hbox{and}\qquad
\frac{2e^{xt}}{e^t+1} = \sum_{n=0}^\infty E_n(x)\frac{t^n}{n!}.
\end{equation}
The first few values are again given in Table~2. These four sequences are 
among the most important special number and polynomial sequences in mathematics,
with numerous applications in number theory, combinatorics, numerical analysis,
and other areas. The most basic properties can be found, e.g., in 
\cite[Ch.~24]{DLMF}.

To set the stage for our results, we first consider the sequence 
${\bf b}=(B_k)_{k\geq 0}$ and compute the first few Hankel determinants
$H_n(B_k)$, $n=0, 1,\ldots, 6$:
\[
1,\quad -\frac{1}{12},\quad -\frac{1}{540},\quad \frac{1}{42\,000},\quad
\frac{1}{3\,215\,625},\quad -\frac{4}{623\,959\,875},\quad
-\frac{64}{213\,746\,467\,935}.
\]
If we factor a somewhat larger term, for instance
\[
H_{10}(B_k) 
= -\frac{2^{42}\cdot3^{15}\cdot 5^4}{11^{11}\cdot 13^9\cdot 17^5\cdot 19^3},
\]
we see, especially in the denominator, that a definite pattern emerges. In fact,
these are special cases of the following known result: If 
${\bf b}=(B_k)_{k\geq 0}$ then for all $n\geq 0$ we have
\begin{equation}\label{1.5}
H_n(B_k) = (-1)^{\binom{n+1}{2}}
\prod_{\ell=1}^n\left(\frac{\ell^{4}}{4(2\ell+1)(2\ell-1)}\right)^{n+1-\ell}.
\end{equation}
It is somewhat surprising that such a formula should exist since the numerators
of Bernoulli numbers are rather deep and mysterious and are, for instance, 
closely related to the classical theory of Fermat's Last Theorem; see, e.g.,
\cite{Ri}. For example, we have
\[
B_{12} = -\frac{691}{2\cdot 3\cdot 5\cdot 7\cdot 13},
\]
where 691 in the numerator is a prime; here we mention in passing that the
denominators of all Bernoulli numbers are completely determined by the theorem
of Clausen and von Staudt; see, e.g., \cite[Sect.~24.10(i)]{DLMF}. However,
the Hankel determinant $H_6(B_k)$, which contains $B_{12}$, has only a power
of 2 in the numerator. 

Another surprising fact becomes apparent when we replace $B_k$ by the Bernoulli
polynomial $B_k(x)$: The determinant \eqref{1.5} remains the same. That is,
there is no dependence on the variable $x$. The smallest nontrivial example of
this is
\[
H_1\big(B_k(x)\big)) = \det\begin{pmatrix}
1 & x-\tfrac{1}{2} \\
x-\tfrac{1}{2} & x^2-x+\tfrac{1}{6}
\end{pmatrix} = -\frac{1}{12}.
\]
This is a well-known phenomenon, which will be mentioned in the next section.

Now, if instead of $b_k=B_k(x)$ we take the subsequences $b_k=B_{2k}(x)$ or
$b_k=B_{2k+1}(x)$, things change drastically. In the first case the Hankel
determinants $H_n(B_{2k}(x))$ are polynomials of increasing degrees that are
apparently irreducible. However, in the second case the Hankel determinants,
while still polynomials in $x$, factor completely into linear factors over
$\mathbb Q$, and a strong pattern emerges. This becomes more visible when we
replace $x$ by $\frac{x+1}{2}$; see Table~1.

\bigskip
\begin{center}
{\renewcommand{\arraystretch}{1.2}
\begin{tabular}{|c||l|}
\hline
$n$ & $H_n(B_{2k+1}(\frac{x+1}{2}))$ \\
\hline
0 & $\frac{1}{2}x$ \\
1 & $-\frac{1}{48}x^2(x^2-1)$ \\
2 & $-\frac{1}{4\,320}x^3(x^2-1)^2(x^2-2^2)$ \\
3 & $\frac{1}{672\,000}x^4(x^2-1)^3(x^2-2^2)^2(x^2-3^2)$ \\
4 & $\frac{1}{102\,900\,000}x^5(x^2-1)^4(x^2-2^2)^3(x^2-3^2)^2(x^2-4^2)$ \\
\hline
\end{tabular}}

\medskip
{\bf Table~1}: $H_n(B_{2k+1}(\frac{x+1}{2}))$ for $0\leq n\leq 4$.
\end{center}
\bigskip
The corresponding general identity is as follows.

\begin{theorem}\label{thm:1.1}
If $b_k=B_{2k+1}(\frac{x+1}{2})$, then for $n\geq 0$ we have
\begin{equation}\label{1.6}
H_n(b_k) = (-1)^{\binom{n+1}{2}}\left(\frac{x}{2}\right)^{n+1}
\prod_{\ell=1}^{n}\left(\frac{\ell^{4}(x^{2}-\ell^{2})}{4(2\ell+1)(2\ell-1)}\right)^{n+1-\ell}.
\end{equation}
\end{theorem}

In this paper we will prove this result and similar identities for
certain sequences of Euler polynomials. In the process we establish some 
mutual connections between Hankel determinants, orthogonal polynomials, and 
certain continued fractions. 

We begin by recalling some basic but important identities
in Section~2, which is followed, in Section~3, by some necessary
background on orthogonal polynomials and continued fractions. Our main results
on Bernoulli and Euler polynomials are then stated and proved in Sections~4
and~5, respectively. Finally, in Section~6 we consider the relationship between 
$H_n(c_k)$ and $H_n(c_{k+1})$, with some consequences for earlier results.

\section{Some basic identities}

In this brief section we collect a few general properties of Bernoulli and 
Euler polynomials and of Hankel determinants that will 
be required in later sections.
We begin with two identities that connect Bernoulli and Euler numbers with their
polynomial analogues:
\begin{align}
B_n(x) &= \sum_{j=0}^n\binom{n}{j}B_j x^{n-j},\label{1.2}\\
E_n(x) &= \sum_{j=0}^n\binom{n}{j}
\frac{E_j}{2^j}\big(x-\tfrac{1}{2}\big)^{n-j}.\label{1.4}
\end{align}
These identities follow easily from \eqref{1.1}, resp.\ \eqref{1.3}. The 
Bernoulli and Euler polynomials are also connected to each other through
\begin{equation}\label{2.0}
E_{n-1}(x) = \frac{2^n}{n}\left(B_n(\tfrac{x+1}{2})-B_n(\tfrac{x}{2})\right)
\end{equation}
(see, e.g., \cite[Eq.~24.4.23]{DLMF}). Another important property is the 
reflection formula
\begin{equation}\label{4.7}
B_{n}(1-x)=(-1)^{n}B_{n}(x)
\end{equation}
(see, e.g., \cite[Eq.~24.4.3]{DLMF}), with the same identity also holding for
the Euler polynomials \cite[Eq.~24.4.4]{DLMF}.

\bigskip
\begin{center}
{\renewcommand{\arraystretch}{1.2}
\begin{tabular}{|r||r|r|l|l|}
\hline
$n$ & $B_n$ & $E_n$ & $B_n(x)$ & $E_n(x)$\\
\hline
0 & 1 & 1 &  1 & 1 \\
1 & $-1/2$ & 0 & $x-\tfrac{1}{2}$ & $x-\tfrac{1}{2}$ \\
2 & $1/6$ & $-1$  & $x^2-x+\tfrac{1}{6}$ & $x^2-x$ \\
3 & 0 & 0 & $x^3-\tfrac{3}{2}x^2+\tfrac{1}{2}x$ & $x^3-\tfrac{3}{2}x^2+\tfrac{1}{4}$ \\
4 & $-1/30$ & 5 & $x^4-2x^3+x^2-\tfrac{1}{30}$ & $x^4-2x^3+x$ \\
5 & 0 & 0 & $x^5-\tfrac{5}{2}x^4+\tfrac{5}{3}x^3-\tfrac{1}{6}x$ &
$x^5-\tfrac{5}{2}x^4+\tfrac{5}{2}x^2-\tfrac{1}{2}$ \\
6 & $1/42$ & $-61$ & $x^6-3x^5+\tfrac{5}{2}x^4-\tfrac{1}{2}x^2+\tfrac{1}{42}$ &
$x^6-3x^5+5x^3-3x$ \\
\hline
\end{tabular}}

\medskip
{\bf Table~2}: $B_n, E_n, B_n(x)$ and $E_n(x)$ for $0\leq n\leq 6$.
\end{center}
\bigskip

Next we state two useful properties of Hankel determinants. We begin with the 
easier one.

\begin{lemma}\label{lem:2.1}
Let $x$ be a variable or a complex number. Then
\begin{equation}\label{2.1}
H_n(x^kc_k) = x^{n(n+1)}H_n(c_k).
\end{equation}
\end{lemma}

\begin{proof}
We consider the determinant of the matrix in \eqref{1.0}, with $x^kc_k$ in place
of $c_k$. We divide the second row by $x$, the third row by $x^2$, etc., and 
finally the $(n+1)$th row by $x^n$. Then, similarly, we divide the 2nd column
by $x$, etc., up to the $(n+1)$th column which we divide by $x^n$. What remains
is the determinant of $(c_{i+j})$, while the total power of $x$ taken out is
$2(1+2+\cdots+n)=n(n+1)$, which completes the proof.
\end{proof}

The next lemma can be found, with proof, in \cite{Ju}; it is also mentioned and
used in various other publications, for instance in \cite[Lemma~15]{Kr}.

\begin{lemma}\label{lem:2.2}
Let $(c_0, c_1,\ldots)$ be a sequence and $x$ a number or a variable. If
\[
c_k(x) = \sum_{j=0}^k\binom{k}{j}c_jx^{k-j},
\]
then for all $n\geq 0$ we have
\begin{equation}\label{2.2}
H_n(c_k(x)) = H_n(c_k).
\end{equation}
\end{lemma}

By the identity \eqref{1.2}, this lemma shows that $H_n(B_k(x)) = H_n(B_k)$,
as already mentioned in the Introduction. Similarly, applying both \eqref{2.2}
and \eqref{2.1} to \eqref{1.4}, we see that
\begin{equation}\label{2.3}
H_n(E_k(x)) = 2^{-n(n+1)}H_n(E_k).
\end{equation}

\section{Orthogonal polynomials and continued fractions}

As already mentioned in the Introduction, the Hankel matrices of 
a sequence are closely related to
certain orthogonal polynomials and continued fractions. The origin of much of
this lies in a remarkable result of Touchard \cite[Eq.~(44)]{To} who defined
a polynomial sequence $(R_n(y))$ by $R_0(y)=1$, $R_1(y)=y+1/2$, and 
\begin{equation}\label{3.1}
R_{n+1}(y) = \big(y+\tfrac{1}{2}\big)R_n(y)+\frac{n^4}{4(2n+1)(2n-1)}R_{n-1}(y)
\qquad (n\geq 1).
\end{equation}
Then Touchard showed that 
\begin{equation}\label{3.2}
y^rR_n(y)\bigg|_{y^k=B_k} = 0\qquad (0\leq r\leq n-1).
\end{equation}
We can compute, for example, 
\[
R_4(y) = y^4+2y^3+\tfrac{17}{7}y^2+\tfrac{10}{7}y+\tfrac{12}{35},
\]
and therefore, by \eqref{3.2}, we get 
\[
B_{4+r}+2B_{3+r}+\tfrac{17}{7}B_{2+r}+\tfrac{10}{7}B_{1+r}+\tfrac{12}{35}B_r = 0
\qquad (0\leq r\leq 3).
\] 
Carlitz \cite[Eq.~(4.7)]{Ca} and more explicitly Al-Salam and Carlitz 
\cite[p.~93]{AC} proved an
analogue of Touchard's result for Euler numbers, and very recently the 
second author and Shi
\cite{JS} extended these results to Bernoulli and Euler polynomials as well 
as higher-order Euler polynomials. 

To explain all this, and to prove the results in this paper, we require some
facts from the classical theory of orthogonal polynomials. Suppose we are given
a sequence ${\bf c}=(c_0, c_1, \ldots)$ of numbers; then it is known that there
exists a positive Borel measure $\mu$ on ${\mathbb R}$ with infinite support
such that 
\begin{equation}\label{3.3} 
c_k = \int_{\mathbb R}y^kd\mu(y),\qquad k = 0, 1, 2, \ldots,
\end{equation}
that is, the moment problem has a solution, if and only if the corresponding
Hankel determinants satisfy $H_n({\bf c})>0$ for all $n\geq 0$. We may also
normalize the sequence such that $c_0=1$. We now summarize several well-known
facts and state them as a lemma with two corollaries; see, e.g., 
\cite[Ch.~2]{Is}.

\begin{lemma}\label{lem:3.1}
If $\mu$ is the measure in \eqref{3.3}, there exists a unique sequence of monic
polynomials $P_n(y)$ of degree $n$, $n=0, 1, \ldots$, and a sequence of positive
numbers $(\zeta_n)_{n\geq 1}$, with $\zeta_0=1$, such that 
\begin{equation}\label{3.4}
\int_{\mathbb R}P_m(y)P_n(y)d\mu(y) = \zeta_n\delta_{m,n},
\end{equation}
where $\delta_{m,n}$ is the Kronecker delta function. Furthermore, for all 
$n\geq 1$ we have $\zeta_n=H_n({\bf c})/H_{n-1}({\bf c})$, and for $n\geq 0$,
\begin{equation}\label{3.5}
P_n(y) = \frac{1}{H_{n-1}({\bf c})}\det
\begin{pmatrix}
c_{0} & c_{1} & \cdots & c_{n}\\
c_{1} & c_{2} & \cdots & c_{n+1}\\
\vdots & \vdots & \ddots & \vdots\\
c_{n-1} & c_{n} & \cdots & c_{2n-1} \\
1 & y & \cdots & y^n
\end{pmatrix},
\end{equation}
where the polynomials $P_n(y)$ satisfy the 3-term recurrence relation 
$P_0(y)=1$, $P_1(y)=y+s_0$, and
\begin{equation}\label{3.6}
P_{n+1}(y) = (y+s_n)P_n(y) - t_nP_{n-1}(y)\qquad (n\geq 1),
\end{equation}
for some sequences $(s_n)_{n\geq 0}$ and $(t_n)_{n\geq 1}$.
\end{lemma}

We now multiply both sides of \eqref{3.5} by $y^r$ and replace $y^j$ by $c_j$, 
which includes replacing the constant term 1 by $c_0$ for $r=0$.
(Similar evaluations apply in the rest of this paper, at $y^k=a_k$ including 
$k=0$, for some sequence $(a_n)$).
Then for $0\leq r\leq n-1$ the last row of the matrix in \eqref{3.5} is 
identical with one of the previous rows, and thus the determinant is 0. When 
$r=n$, the determinant is $H_n({\bf c})$. We therefore have the following 
result.

\begin{corollary}\label{cor:3.2}
With the sequence $(c_k)$ and the polynomials $P_n(y)$ as above, we have
\begin{equation}\label{3.7}
y^rP_n(y)\bigg|_{y^k=c_k} = \begin{cases}
0 &\hbox{when}\;\; 0\leq r\leq n-1,\\
H_n({\bf c})/H_{n-1}({\bf c}) &\hbox{when}\;\; r=n.
\end{cases}
\end{equation}
\end{corollary}

This corollary, by the way, is consistent with \eqref{3.2}. Another important
consequence of Lemma~3.1 will be an essential ingredient in most of our results.
For the sake of completeness we give a proof of this well-known result.

\begin{corollary}\label{cor:3.3}
With the sequence $(t_n)$ as in \eqref{3.6}, we have
\begin{equation}\label{3.8}
H_n({\bf c}) = t_1^nt_2^{n-1}\cdots t_{n-1}^2t_n\qquad (n\geq 0).
\end{equation}
\end{corollary}

\begin{proof}
Multiplying both sides of \eqref{3.6} by $P_{n-1}(y)$, then integrating and using
\eqref{3.4}, we get the recurrence $\zeta_n=t_n\zeta_{n-1}$. Iterating this and
recalling that $\zeta_0=1$, we get $\zeta_n=t_{n}t_{n-1}\cdots t_1$. If we 
combine this with the identity $H_n({\bf c})=\zeta_n H_{n-1}({\bf c})$ (see the
line before \eqref{3.5}) and iterate again, we immediately get \eqref{3.8}.
\end{proof} 

Corollary~\ref{cor:3.3} shows, in particular, that the Hankel determinants in
\eqref{1.5} immediately follow from the recurrence relation \eqref{3.1}. This
example brings the following issue to light:

The Hankel determinants in \eqref{1.5} are obviously not all positive; but this
was a requirement in the theory involving orthogonal polynomials. To get around
this potential problem, we consider $\widetilde{B}_k=i^kB_k$ for all $k\geq 0$,
where $i=\sqrt{-1}$. Then (see Table~2) we have $\widetilde{B}_0=1$,
$\widetilde{B}_1=-i/2$, and $\widetilde{B}_{2j}$ is negative for all $j$, while
all corresponding Hankel determinants are positive, as required. Finally, by
Lemma~\ref{lem:2.1} with $x=i$ we have
\[
H_n(\widetilde{B}_k) = (-1)^{n(n+1)/2}H_n(B_k),
\]
which is consistent with \eqref{1.5}. Observations of this kind can also be 
made in other similar situations in this paper, so that we do not need to worry
about the positivity of the related Hankel determinants.

The next result which we require establishes a connection with certain continued
fractions (in this case called $J$-fractions). It can be found in various 
relevant publications, for instance in \cite[p.~20]{Kr}.

\begin{lemma}\label{lem:3.4}
Let ${\bf c}=(c_k)_{k\geq 0}$ be a sequence of numbers with $c_0\neq 0$, and
suppose that its generating function is written in the form
\begin{equation}\label{3.9}
\sum_{k=0}^{\infty}c_k t^k
= \cfrac{c_0}{1+s_0t-\cfrac{t_1t^2}{1+s_1t-\cfrac{t_2t^2}{1+s_2t-\ddots}}},
\end{equation}
where both sides are considered as formal power series. Then
the sequences $(s_n)$ and $(t_n)$ are the same as in \eqref{3.6}, and
\begin{equation}\label{3.10}
H_n({\bf c}) = c_0^{n+1}t_1^{n}t_2^{n-1}\cdots t_{n-1}^2t_{n}.
\end{equation}
\end{lemma}

With the exception of the factor $c_0^{n+1}$, the identities \eqref{3.10} and 
\eqref{3.8} are the same. The difference comes from the assumption $c_0=1$ in
\eqref{3.3} and in Lemma~\ref{lem:3.1}, which can be suitably relaxed.

We finish this section with some definitions and results which will also be 
needed later.  Following the
usage in books such as \cite{Cu} or \cite{LW}, we write
\begin{equation}\label{3.11}
b_0+\operatornamewithlimits{\bf K}_{m=1}^{\infty}\big(a_m/b_m\big)
=b_0+\operatornamewithlimits{\bf K}\big(a_m/b_m\big)
=b_0+\cfrac{a_1}{b_1+\cfrac{a_2}{b_2+\ddots}}
\end{equation}
for an infinite continued fraction. The $n$th {\it approximant\/} is
expressed by
\begin{equation}\label{3.12}
b_0+\operatornamewithlimits{\bf K}_{m=1}^{n}\big(a_m/b_m\big)
=b_0+\cfrac{a_1}{b_1+\ddots+\cfrac{a_n}{b_n}} = \frac{A_n}{B_n},
\end{equation}
and $A_n, B_n$ are called the $n$th {\it numerator} and {\it denominator},
respectively. The continued fraction \eqref{3.11} is said to {\it converge} if
the sequence of approximants in \eqref{3.12} converges. In this case, the limit
is called the {\it value} of the continued fraction \eqref{3.11}.

Two continued fractions are said to be {\it equivalent} if and only if they
have the same sequences of approximants. In other words, we have
\[
b_0+\operatornamewithlimits{\bf K}_{m=1}^{n}\big(a_m/b_m\big)
=d_0+\operatornamewithlimits{\bf K}_{m=1}^{n}\big(c_m/d_m\big)
\]
if and only if there exists a sequence of nonzero complex numbers
$(r_m)_{m\geq 0}$ with $r_0=1$, such that
\begin{equation}\label{3.13}
d_m = r_mb_m,\qquad c_{m+1}=r_{m+1}r_ma_{m+1}\qquad (m\geq 0);
\end{equation}
see \cite[Eq.~(1.4.2)]{Cu}. We also require the following special case of the
more general concept of a contraction; see, e.g., \cite[p.~16]{Cu}.

\begin{definition}\label{def:3.5}
{\rm Let $A_n, B_n$ be the $n$th numerator and denominator, respectively, of a
continued fraction cf$_1:=b_0+{\bf K}\big(a_m/b_m\big)$, and let $C_n, D_n$ be
the corresponding quantities of cf$_2:=d_0+{\bf K}\big(c_m/d_m\big)$. Then
Then cf$_2$ is called an {\it even canonical contraction} of cf$_1$ if
\[
C_n=A_{2n},\qquad D_n=B_{2n}\qquad (n\geq 0),
\]
and is called an {\it odd canonical contraction} of cf$_1$ if}
\[
C_0=\frac{A_1}{B_1},\quad D_0=1,\quad
C_n=A_{2n+1},\quad D_n=B_{2n+1}\qquad (n\geq 1).
\]
\end{definition}

We will now state three identities that will be used in later sections; see
\cite[pp.~16--18]{Cu}, \cite[pp.~83--85]{LW}, or \cite[pp.~21--22]{Wa} for proofs and further
details.

\begin{lemma}\label{lem:3.6}
An even canonical contraction of $b_0+{\bf K}\big(a_m/b_m\big)$ exists if and
only if $b_{2k}\neq 0$ for $k\geq 1$, and we have
\begin{align}
b_0&+\operatornamewithlimits{\bf K}_{m=1}^{\infty}\big(a_m/b_m\big)=b_0\label{3.14}\\
&+\cfrac{a_{1}b_{2}}{a_{2}+b_{1}b_{2}-\cfrac{a_{2}a_{3}b_{4}/b_{2}}
{a_{4}+b_{3}b_{4}+a_{3}b_{4}/b_{2}-\cfrac{a_{4}a_{5}b_{6}/b_{4}}
{a_{6}+b_{5}b_{6}+a_{5}b_{6}/b_{4}-\cfrac{a_{6}a_{7}b_{8}/b_{6}}{\ddots}}}}.\nonumber
\end{align}
In particular, with $b_0=0$, $b_k=1$ for $k\geq 1$, $a_1=1$, and
$a_k=\alpha_{k-1}t$ ($k\geq 1$), for some variable $t$, we have
\begin{equation}\label{3.15}
\cfrac{1}{1-\cfrac{\alpha_{1}t}{1-\cfrac{\alpha_{2}t}{1-\ddots}}}
=\cfrac{1}{1-\alpha_{1}t-\cfrac{\alpha_{1}\alpha_{2}t^{2}}{1-(\alpha_{2}
+\alpha_{3})t-\cfrac{\alpha_{3}\alpha_{4}t^{2}}{1-(\alpha_{4}+\alpha_{5})t
-\cfrac{\alpha_{5}\alpha_{6}t^{2}}{\ddots}}}}.
\end{equation}
Similarly, an odd canonical contraction gives
\begin{equation}\label{3.16}
1+\cfrac{\alpha_{1}t}{1-(\alpha_{1}+\alpha_{2})t-\cfrac{\alpha_{2}\alpha_{3}t^{2}}
{1-(\alpha_{3}+\alpha_{4})t-\cfrac{\alpha_{4}\alpha_{5}t^{2}}
{1-(\alpha_{5}+\alpha_{6})t-\cfrac{\alpha_{6}\alpha_{7}t^{2}}{\ddots}}}}
\end{equation}
for the continued fraction on the left-hand side of \eqref{3.15}.
\end{lemma}

\section{The Bernoulli polynomial case}

We begin by stating the main result of this section.

\begin{theorem}\label{thm:4.1}
Let $b_{k}=B_{2k+1}(\frac{x+1}{2})$, and let $\big(W_{n}(y;x)\big)_{n\geq 0}$
be the sequence of polynomials orthogonal in $y$ with respect to the sequence
$(b_{k})$, that is,
\begin{equation}\label{4.1}
y^{r}W_{n}(y;x)\bigg|_{y^k=b_{k}}=0\qquad (0\leq r\leq n-1).
\end{equation}
Then we have $W_{0}(y;x)=1$, $W_{1}(y;x)=y+\sigma_{0}$, and for $n\geq 1$, 
\begin{equation}\label{4.2}
W_{n+1}(y;x)=(y+\sigma_{n})W_{n}(y;x)+\tau_{n}W_{n-1}(y;x),
\end{equation}
where
\begin{equation}\label{4.3}
\sigma_{n}=\binom{n+1}{2}-\frac{x^{2}-1}{4}\qquad\hbox{and}\qquad
\tau_{n}=\frac{n^{4}(x^{2}-n^{2})}{4(2n+1)(2n-1)}.
\end{equation}
\end{theorem}

Since $b_0=B_1(\frac{x+1}{2})=\frac{x}{2}$ (see Table~2), Lemma~\ref{lem:3.4}
with $c_0=x/2$, $s_j=\sigma_j$, and $t_j=-\tau_j$ immediately gives 
Theorem~\ref{thm:1.1} as a corollary.

By Lemma~\ref{lem:3.4}, for the proof of Theorem~\ref{thm:4.1} it suffices to
prove the following lemma. On the left-hand side we will have a formal power
series which could also be interpreted as an asymptotic expansion.

\begin{lemma}\label{lem:4.2}
We have the continued fraction expansion
\begin{equation}\label{4.4}
\sum_{k=0}^{\infty}B_{2k+1}(\tfrac{x+1}{2})z^{2k}
=\cfrac{\tfrac{x}{2}}{1+\sigma_0z^2+\cfrac{\tau_1z^4}{1+\sigma_1z^2+\cfrac{\tau_2z^4}{1+\sigma_2z^2+\ddots}}}.
\end{equation}
\end{lemma}

An important tool for proving Lemma~\ref{lem:4.2}, as well as the results 
in the next section, is the 
{\it polygamma function\/}. For an integer $n\geq 0$ it is defined by
\[
\psi^{(n)}(z)
:=\frac{\mathrm{d}^{n+1}}{\mathrm{d}z^{n+1}}\big(\log\Gamma(z)\big),
\]
where $\Gamma(z)$ is the gamma function. For $n=0$ we have
$\psi^{(0)}(z)=\psi(z)=\Gamma'(z)/\Gamma(z)$, the well-known and important
{\it digamma function\/}, and $\psi^{(1)}(z)=\psi'(z)$ is sometimes called the
{\it trigamma function\/}.

\begin{lemma}\label{lem:4.3}
We have the formal power series
\begin{equation}\label{4.5}
\sum_{k=0}^{\infty}B_{2k+1}(\tfrac{x+1}{2})z^{2k} = \frac{1}{2z^{2}}
\left(\psi'\big(\tfrac{1}{z}+\tfrac{1-x}{2}\big)
-\psi'\big(\tfrac{1}{z}+\tfrac{1+x}{2}\big)\right).
\end{equation}
\end{lemma}

\begin{proof}
We use the following well-known complete asymptotic expansion, valid for
$\left|\arg z\right|<\pi$:
\[
\log\Gamma(z+x)=\left(z+x-\tfrac{1}{2}\right)\log z-z+\frac{\log(2\pi)}{2}+\sum_{n=1}^{\infty}\frac{(-1)^{n+1}B_{n+1}(x)}{n(n+1)z^{n}};
\]
see, e.g., \cite[p.~48, Eq.~(12)]{Er}. Differentiating twice with respect to
$z$, we get
\begin{equation}\label{4.6}
\psi(z+x)=\log z+\frac{x-\frac{1}{2}}{z}
-\sum_{n=1}^{\infty}\frac{(-1)^{n+1}B_{n+1}(x)}{(n+1)z^{n+1}},
\end{equation}
and 
\[
\psi'(z+x)=\frac{1}{z}-\frac{x-\frac{1}{2}}{z^{2}}
+\sum_{n=2}^{\infty}\frac{(-1)^{n}B_{n}(x)}{z^{n+1}}.
\]
Therefore, 
\begin{align*}
\psi'(z+\tfrac{1+x}{2})-\psi'(z+\tfrac{1-x}{2}) 
& =\frac{1}{z}-\frac{x/2}{z^{2}}
+\sum_{n=2}^{\infty}\frac{(-1)^{n}B_{n}(\frac{1+x}{2})}{z^{n+1}}\\
&\qquad-\left(\frac{1}{z}-\frac{-x/2}{z^{2}}
+\sum_{n=2}^{\infty}\frac{(-1)^{n}B_{n}(\frac{1-x}{2})}{z^{n+1}}\right).\\
 & =-\frac{x}{z^{2}}+\sum_{n=2}^{\infty}\frac{(-1)^{n}}{z^{n+1}}
\left(B_{n}(\tfrac{1+x}{2})-B_{n}(\tfrac{1-x}{2})\right).
\end{align*}
Using the reflection formula \eqref{4.7}, we see
\[
B_{n}(\tfrac{1-x}{2})=B_{n}(1-\tfrac{1+x}{2})=(-1)^{n}B_{n}(\tfrac{1+x}{2}).
\]
This, finally, implies 
\[
\frac{1}{2}\left(\psi'(z+\tfrac{1+x}{2})-\psi'(z+\tfrac{1-x}{2})\right)
=-\frac{x}{2z^{2}}-\sum_{n=1}^{\infty}\frac{B_{2n+1}(\frac{1+x}{2})}{z^{2n+2}}.
\]
We recall that $B_{1}(\frac{1+x}{2})=\frac{x}{2}$. The change of variables 
$z\mapsto1/z$ then yields
\[
\frac{\psi'(\frac{1}{z}+\frac{1+x}{2})-\psi'(\frac{1}{z}+\frac{1-x}{2})}{2z^{2}}
=-\sum_{n=0}^{\infty}B_{2n+1}(\tfrac{1+x}{2})z^{2n}
\]
which completes the proof.
\end{proof}

\begin{proof}[Proof of Lemma~\ref{lem:4.2}]
We denote the right-hand side of \eqref{4.5} by $F(z)$, where the dependence 
on $x$ is implied. Using the series representation
\[
\psi'(z)=\sum_{k=0}^{\infty}\frac{1}{(z+k)^{2}}
\]
(see, e.g., \cite[Eq.~5.15.1]{DLMF}), we rewrite $F(z)$ as
\begin{equation}\label{4.7a}
2z^2F(z) = \sum_{k=0}^{\infty}
\left(\frac{1}{\left(\frac{1}{z}+\frac{1-x}{2}+k\right)^{2}}
-\frac{1}{\left(\frac{1}{z}+\frac{1+x}{2}+k\right)^{2}}\right),
\end{equation}
Now we use the following continued fraction due to Ramanujan; see
\cite[p.~158]{Be} or \cite[pp.~591--592]{LW}:
\begin{align}
\sum_{k=0}^{\infty}&\frac{1}{(s-b+2k+1)^{2}}
-\sum_{k=0}^{\infty}\frac{1}{(s+b+2k+1)^{2}}\label{4.8}\\
&=\cfrac{b}{1(s^{2}-b^{2}+1)-\cfrac{4(1^{2}-b^{2})1^{4}}{3(s^{2}-b^{2}+5)
-\cfrac{4(2^{2}-b^{2})2^{4}}{5(s^{2}-b^{2}+13)-\ddots}}},\nonumber
\end{align}
where the sequence of constants 1, 5, 13 is given by $(2n(n+1)+1)_{n\geq 0}$.
Dividing both sides of \eqref{4.7a} by 4, we see that the right-hand side of
\eqref{4.7a}, with $s=2/z$ and $b=x$, is identical with the left-hand side of
\eqref{4.8}. Therefore,
\[
\frac{z^{2}F(z)}{2} 
=\cfrac{x}{1\left(\frac{4}{z^{2}}-x^{2}+1\right)-\cfrac{4(1^{2}-x^{2})1^{4}}
{3\left(\frac{4}{z^{2}}-x^{2}+5\right)-\cfrac{4(2^{2}-x^{2})2^{4}}
{5\left(\frac{4}{z^{2}}-x^{2}+13\right)-\ddots}}}.
\]
Using equivalence of continued fractions with \eqref{3.13}, where
$r_{m}= z^{2}/(4(2m-1))$,
we get
\begin{equation}\label{4.9}
\frac{z^{2}F(z)}{2}
=\cfrac{\frac{xz^{2}}{4}}{1+\frac{(1-x^{2})}{4}z^{2}-\cfrac{\frac{(1^{2}-x^{2})1^{4}}{4\cdot1\cdot3}z^{4}}
{1+\frac{(5-x^{2})}{4}z^{2}-\cfrac{\frac{(2^{2}-x^{2})2^{4}}{4\cdot3\cdot5}z^{4}}{1+\frac{(13-x^{2})}{4}z^{2}-\ddots}}}.
\end{equation}
Finally, recalling that
\[
B_{1}(\tfrac{1+x}{2})=\frac{x}{2},\qquad 
-\frac{n^{4}(n^{2}-x^{2})}{4(2n-1)(2n+1)}=\tau_{n}\quad (n\geq 1),
\]
and
\[
\frac{2n(n+1)+1-x^{2}}{4}=\binom{n+1}{2}-\frac{x^{2}-1}{4}=\sigma_{n}\quad
(n\geq 0),
\]
we see that \eqref{4.9} immediately gives \eqref{4.4}, which completes the 
proof.
\end{proof}

Theorems~\ref{thm:1.1} and \ref{thm:4.1} give rise to the natural question
whether there exist similar results for {\it even-index\/} Bernoulli 
polynomials. This question is also related to the fact that there is the known
identity 
\[
H_n\big(B_{2k}(\tfrac{1}{2})\big) = \prod_{\ell=1}^n
\left(\frac{\ell^4(2\ell-1)^4}{(4\ell-3)(4\ell-1)^2(4\ell+1)}\right)^{n-\ell+1},
\]
which is due to Chen \cite[Eq.~(41)]{Ch}. However, as far as Bernoulli
polynomials are concerned, we have
\begin{align*}
H_1\big(B_{2k}(\tfrac{x+1}{2})\big) &= -\tfrac{1}{12}x^2+\tfrac{1}{45},\\
H_2\big(B_{2k}(\tfrac{x+1}{2})\big) 
&= -\tfrac{1}{540}x^6+\tfrac{97}{18\,900}x^4-\tfrac{11}{4\,725}x^2
+\tfrac{16}{55\,125},
\end{align*}
and in general $H_n\big(B_{2k}(\tfrac{x+1}{2})\big)$ seems to be an irreducible
polynomial of degree $n(n+1)$. We did not pursue this question any further.

For some additional comments on $H_n\big(B_{2k+\nu}(\tfrac{x+1}{2})\big)$ for
other positive integers $\nu$, see the end of Section~6.

\section{The Euler polynomial cases}

For the Euler numbers and polynomials, defined in \eqref{1.3}, Al-Salam and 
Carlitz showed in \cite[Eq.~(4.2), (5.2)]{AC} that
\begin{equation}\label{5.1}
H_n(E_k) = (-1)^{\binom{n+1}{2}}\prod_{\ell=1}^n\big(\ell!\big)^2,\qquad
H_n(E_k(x)) = (-\tfrac{1}{4})^{\binom{n+1}{2}}\prod_{\ell=1}^n\big(\ell!\big)^2.
\end{equation}
By \eqref{2.3}, the second identity in \eqref{5.1} follows 
from the first one. Furthermore, in
analogy to the orthogonal polynomials $R_n(x)$ in \eqref{3.1}, Al-Salam and
Carlitz obtained the monic orthogonal polynomial with respect to the Euler
numbers: Let $Q_0(y)=1$, $Q_1(y)=~y$, and
\begin{equation}\label{5.2}
Q_{n+1}(y) = y\,Q_n(y) + n^2\,Q_{n-1}(y)\qquad (n\geq 1).
\end{equation}
Then
\[
y^rQ_n(y)\bigg|_{y^k=E_k} = 0\qquad(0\leq r\leq n-1);
\]
see also Corollary~\ref{cor:3.2}. We also note that \eqref{5.2} and 
Corollary~\ref{cor:3.3} give the first identity in \eqref{5.1}, after some 
easy manipulation. All 
this was recently extended by the second author and Shi to higher-order Euler
polynomials, of which the ordinary Euler polynomials are a special case.

The main results of this section are analogous to Theorems~\ref{thm:4.1} and
\ref{thm:1.1} in that we will deal only with even-index polynomials or 
odd-index polynomials. However, in contrast to those results, here we will have
three different but related cases.

\begin{theorem}\label{thm:5.1}
For $\nu=0, 1, 2$, let $c_k^{(\nu)}:=E_{2k+\nu}(\frac{x+1}{2})$, and let
$\big(q_n^{(\nu)}(y;x)\big)_{n\geq 0}$ be the sequence of  monic polynomials
orthogonal in $y$ with respect to the sequence $(c_k^{(\nu)})$, that is,
\begin{equation}\label{5.3}
y^{r}q_{n}^{(\nu)}(y;x)\bigg|_{y^k=c_k^{(\nu)}}=0\qquad (0\leq r\leq n-1).
\end{equation}
Then we have $q_{0}^{(\nu)}=1$, $q_{1}^{(\nu)}(y;x)=y+\sigma_{0}^{(\nu)}$,
and for $n\geq 1$,
\begin{equation}\label{5.4}
q_{n+1}^{(\nu)}(y;x)
=(y+\sigma_{n}^{(\nu)})q_{n}^{(\nu)}(y;x)+\tau_{n}^{(\nu)}q_{n-1}^{(\nu)}(y;x),
\end{equation}
where for $\nu=0, 1, 2$,
\begin{equation}\label{5.5}
\sigma_{n}^{(\nu)}=(2n+1)(n+\tfrac{\nu}{2})-\frac{x^{2}-1}{4},\qquad
\tau_{n}^{(\nu)}=\frac{n^{2}}{4}\left(x^{2}-(2n+\nu-1)^{2}\right).
\end{equation}
\end{theorem}

Now Lemma~\ref{lem:3.4} with $c_0=c_0^{(\nu)}=E_{\nu}(\frac{x+1}{2})$, with
$s_j=\sigma_j^{(\nu)}$, and $t_j=-\tau_j^{(\nu)}$ immediately gives the 
following Hankel determinants.

\begin{corollary}\label{cor:5.2}
Let $c_k^{(\nu)}=E_{2k+\nu}(\frac{x+1}{2})$ for $\nu=0, 1, 2$. Then for all
$n\geq 0$ we have
\begin{equation}\label{5.6}
H_n(c_k^{(\nu)})=(-1)^{\binom{n+1}{2}}E_{\nu}(\tfrac{x+1}{2})^{n+1}
\prod_{\ell=1}^n\left(\tau_{\ell}^{(\nu)}\right)^{n+1-\ell},
\end{equation}
or more explicitly,
\begin{align}
H_n(c_k^{(0)}) &= (-1)^{\binom{n+1}{2}}\prod_{\ell=1}^n
\left(\frac{\ell^2}{4}\big(x^2-(2\ell-1)^2\big)\right)^{n+1-\ell},\label{5.7}\\
H_n(c_k^{(1)}) &= (-1)^{\binom{n+1}{2}}\left(\frac{x}{2}\right)^{n+1}
\prod_{\ell=1}^n\left(\frac{\ell^2}{4}\big(x^2-(2\ell)^2\big)\right)^{n+1-\ell},\label{5.8}\\
H_n(c_k^{(2)}) &= (-1)^{\binom{n+1}{2}}\left(\frac{x^{2}-1}{4}\right)^{n+1}
\prod_{\ell=1}^{n}\left(\frac{\ell^{2}}{4}(x^{2}-(2\ell+1)^{2})\right)^{n+1-\ell}.\label{5.9}
\end{align}
\end{corollary}

To obtain \eqref{5.7}--\eqref{5.9} from \eqref{5.6}, we only need to  notice
that
\[
E_{0}(\tfrac{x+1}{2})=1,\qquad E_{1}(\tfrac{x+1}{2})=\frac{x}{2}, \qquad
E_{2}(\tfrac{x+1}{2})=\frac{x^2-1}{4};
\]
see Table~2.

The proof of Theorem~\ref{thm:5.1} will be similar in nature to that of
Theorem~\ref{thm:4.1}. In particular, by Lemma~\ref{lem:3.4} it suffices to 
prove the following lemma.

\begin{lemma}\label{lem:5.3}
For $\nu=0, 1, 2$ we define the formal power series
\begin{equation}\label{5.10}
F^{(\nu)}(z) := \sum_{k=0}^{\infty}E_{2k+\nu}(\tfrac{x+1}{2})z^{2k},
\end{equation}
where the dependence on $x$ is implied. Then we have
\begin{equation}\label{5.10a}
F^{(\nu)}(z)
=\cfrac{E_{\nu}(\tfrac{x+1}{2})}{1+\sigma_{0}^{(\nu)}z^{2}+\cfrac{\tau_{1}^{(\nu)}z^{4}}{1+\sigma_{1}^{(\nu)}z^{2}+\cfrac{\tau_{2}^{(\nu)}z^{4}}{\ddots}}}.
\end{equation}
\end{lemma}

The proof of this, in turn, relies on the following important connection with
the digamma function.

\begin{lemma}\label{lem:5.4}
With $F^{(\nu)}(z)$, $\nu=0, 1, 2$, as defined in \eqref{5.10}, we have
\begin{align}
F^{(0)}(z) &= \frac{\psi(\tfrac{1}{2z}+\tfrac{3+x}{4})
-\psi(\tfrac{1}{2z}+\tfrac{1+x}{4})+\psi(\tfrac{1}{2z}+\tfrac{3-x}{4})
-\psi(\tfrac{1}{2z}+\tfrac{1-x}{4})}{2z},\label{5.11} \\
F^{(1)}(z) &= \frac{-\psi(\tfrac{1}{2z}+\tfrac{3+x}{4})
+\psi(\tfrac{1}{2z}+\tfrac{1+x}{4})+\psi(\tfrac{1}{2z}+\tfrac{3-x}{4})
-\psi(\tfrac{1}{2z}+\tfrac{1-x}{4})}{2z^2},\label{5.12}\\
F^{(2)}(z) &= \frac{F^{(0)}(z)-1}{z^2}.\label{5.12a}
\end{align}
\end{lemma}

\begin{proof}
We use the asymptotic expansion \eqref{4.6} and replace $x$ by $(3+x)/4$ and by
$(1+x)/4$. The upon subtracting, we get
\begin{equation}\label{5.13}
\sum_{n=1}^{\infty}(-1)^{n+1}
\frac{B_{n+1}(\tfrac{3+x}{4})-B_{n+1}(\tfrac{1+x}{4})}{(n+1)z^{n+1}}
=-\psi(z+\tfrac{3+x}{4})+\psi(z+\tfrac{1+x}{4})+\frac{1}{2z}.
\end{equation}
By the identity \eqref{4.7}, we have
\[
(-1)^{n+1}B_{n+1}(\tfrac{3+x}{4}) = B_{n+1}(\tfrac{1-x}{4}),\qquad
(-1)^{n+1}B_{n+1}(\tfrac{1+x}{4}) = B_{n+1}(\tfrac{3-x}{4}).
\]
Applying these identities to the left-hand side of \eqref{5.13} and then 
replacing $x$ by $-x$ and multiplying both sides by $-1$, we get
\begin{equation}\label{5.14}
\sum_{n=1}^{\infty}
\frac{B_{n+1}(\tfrac{3+x}{4})-B_{n+1}(\tfrac{1+x}{4})}{(n+1)z^{n+1}}
=\psi(z+\tfrac{3-x}{4})-\psi(z+\tfrac{1-x}{4})-\frac{1}{2z}.
\end{equation}

Now, by \eqref{2.0} we have
\begin{equation}\label{5.14a}
E_{2n+\nu}(\tfrac{x+1}{2}) = \frac{2^{2n+\nu+1}}{2n+\nu+1}\left(
B_{2n+\nu+1}(\tfrac{3+x}{4})-B_{2n+\nu+1}(\tfrac{1+x}{4})\right),
\end{equation}
and thus
\begin{align*}
2\sum_{n=0}^{\infty}&E_{2n+\nu}(\tfrac{x+1}{2})\left(\tfrac{1}{2z}\right)^{2n+\nu+1} \\
&=2\sum_{n=0}^{\infty}\frac{B_{2n+\nu+1}(\tfrac{3+x}{4})}{(2n+\nu+1)z^{2n+\nu+1}}
-2\sum_{n=0}^{\infty}\frac{B_{2n+\nu+1}(\tfrac{1+x}{4})}{(2n+\nu+1)z^{2n+\nu+1}} \\
&=\sum_{n=0}^{\infty}\frac{B_{n+\nu+1}(\tfrac{3+x}{4})}{(n+\nu+1)z^{n+\nu+1}}
+\sum_{n=0}^{\infty}\frac{(-1)^nB_{n+\nu+1}(\tfrac{3+x}{4})}{(n+\nu+1)z^{n+\nu+1}}\\
&\qquad-\sum_{n=0}^{\infty}\frac{B_{n+\nu+1}(\tfrac{1+x}{4})}{(n+\nu+1)z^{n+\nu+1}}
-\sum_{n=0}^{\infty}\frac{(-1)^nB_{n+\nu+1}(\tfrac{1+x}{4})}{(n+\nu+1)z^{n+\nu+1}}.
\end{align*}
Therefore
\begin{align}
2\sum_{n=0}^{\infty}E_{2n+\nu}(\tfrac{x+1}{2})&\left(\tfrac{1}{2z}\right)^{2n+\nu+1}
=\sum_{n=\nu}^{\infty}
\frac{B_{n+1}(\tfrac{3+x}{4})-B_{n+1}(\tfrac{1+x}{4})}{(n+1)z^{n+1}}\label{5.15}\\
&-(-1)^{\nu}\sum_{n=\nu}^{\infty}(-1)^{n+1}
\frac{B_{n+1}(\tfrac{3+x}{4})-B_{n+1}(\tfrac{1+x}{4})}{(n+1)z^{n+1}}.\nonumber
\end{align}
First we let $\nu=1$. We use \eqref{5.13} and \eqref{5.14}, replace $z$ by
$\frac{1}{2z}$, and divide both sides of \eqref{5.15} by $2z^2$, to get 
\eqref{5.12}.

When $\nu=0$, we need to subtract the terms for $n=0$ from the right-hand side
of \eqref{5.15}, namely
\[
\frac{2}{z}\left(B_{1}(\tfrac{3+x}{4})-B_{1}(\tfrac{1+x}{4})\right)
=\frac{2}{z}\left(\left(\frac{3+x}{4}-\frac{1}{2}\right)
-\left(\frac{1+x}{4}-\frac{1}{2}\right)\right)
=\frac{1}{z}.
\]
Now we use again \eqref{5.13} and \eqref{5.14} and note that this last 
expression $\frac{1}{z}$ and the terms $\pm\frac{2}{z}$ cancel each other.
Once again, by replacing $z$ by $\frac{1}{2z}$ and dividing both sides by $2z$, 
we get \eqref{5.11}. Finally, \eqref{5.12a} follows directly from \eqref{5.10}.
\end{proof}

\begin{proof}[Proof of Lemma~\ref{lem:5.3}]
We begin with the case $\nu=0$ and use \eqref{5.11}. By Equations (T.4) and
(T.6) in \cite[p.~274]{La} we have
\begin{equation}\label{5.18}
\frac{\psi\left(\tfrac{s-a+3b}{4b}\right)-\psi\left(\tfrac{s-a+b}{4b}\right)
+\psi\left(\tfrac{s+a+3b}{4b}\right)-\psi\left(\tfrac{s+a+b}{4b}\right)}{4b}
=\cfrac{1}{s+\cfrac{a_{1}}{s+\cfrac{a_{2}}{s+\ddots}}},
\end{equation}
where for $n\geq 1$, $a_{2n-1}=(2n-1)^{2}b^{2}-a^{2}$ and $a_{2n}=4n^{2}b^{2}.$ If we set
$b=1$, $a=x$, and $s=2/z$ then the left-hand side of \eqref{5.18} is
exactly $zF^{(0)}(z)/2$, so that
\begin{equation}\label{5.19}
F^{(0)}(z)
=\cfrac{\frac{2}{z}}{\frac{2}{z}+\cfrac{a_{1}}{\frac{2}{z}+\cfrac{a_{2}}
{\frac{2}{z}+\ddots}}}
=\cfrac{1}{1+\cfrac{\frac{a_{1}}{4}z^{2}}{1+\cfrac{\frac{a_{2}}{4}z^{2}}
{1+\ddots}}},
\end{equation}
where for the equation on the right we have used \eqref{3.13} with $r_{m}=z/2$.
Next we apply \eqref{3.15} with $t=z^{2}$ and $\alpha_{n}=-a_{n}/4$.
It is now easy to check that
\[
-\alpha_{1}=\sigma_{0}^{(0)},\quad 
-\alpha_{2n-1}\alpha_{2n}=\tau_{n}^{(0)},\quad\text{and}\quad 
-\alpha_{2n}-\alpha_{2n+1}=\sigma_{n}^{(0)}\quad (n\geq 1).
\]
The identity \eqref{5.19} therefore gives \eqref{5.10a} for $\nu=0$.

Next, we consider the case $\nu=1$ and use \eqref{5.12}. By Equations (U.4) and
(U.7) in \cite[p.~275]{La} we have
\begin{equation}\label{5.20}
\frac{\psi\left(\tfrac{s-a+3}{4}\right)-\psi\left(\tfrac{s+a+3}{4}\right)
+\psi\left(\tfrac{s+a+1}{4}\right)-\psi\left(\tfrac{s-a+1}{4}\right)}{4}
=\cfrac{\frac{a}{s^{2}-1}}{1+\cfrac{\frac{2^{2}-a^{2}}{s^{2}-1}}
{1+\cfrac{\frac{2^{2}}{s^{2}-1}}{1+\ddots}}}.
\end{equation}
If we set $a=x$ and $s=2/z$, then the left-hand side of \eqref{5.20} becomes 
$z^{2}F^{(1)}(z)/2$, so that
\[
F^{(1)}(z)
=\cfrac{\frac{\frac{2x}{z}}{\frac{4}{z^{2}}-1}}{1+\cfrac{\frac{2^{2}-a^{2}}
{\frac{4}{z^{2}}-1}}{1+\cfrac{\frac{2^{2}}{\frac{4}{z^{2}}-1}}{1+\ddots}}}.
\]
Now we use \eqref{3.14} with $b_{n}=1$ for all $n\geq 0$, and with
\[
a_{1}=\frac{2xt^{2}}{\frac{4}{z^{2}}-1}=\frac{2xt^{2}z^{2}}{4-z^{2}}
\]
and for $n\geq 1$
\[
a_{2n}=\frac{(2n)^{2}-x^{2}}{\frac{4}{z^{2}}-1}
=\frac{(2n)^{2}-x^{2}}{4-z^{2}}z^{2},\qquad
a_{2n+1}=\frac{4n^{2}}{\frac{4}{z^{2}}-1}=\frac{4n^{2}z^{2}}{4-z^{2}}.
\]
Then 
\[
-a_{2n}a_{2n+1}\frac{b_{2n+2}}{b_{2n}}
=\frac{4n^{2}(x^{2}-(2n)^{2})}{\left(\frac{4}{z^{2}}-1\right)^{2}}
=\frac{4n^{2}(x^{2}-(2n)^{2})z^{4}}{(4-z^{2})^{2}},
\]
and 
\begin{align*}
a_{2n+2}+b_{2n+1}b_{2n+2}+a_{2n+1}\frac{b_{2n+2}}{b_{2n}}
&=1+\frac{4(n+1)^{2}+4n^{2}-x^{2}}{\frac{4}{z^{2}}-1} \\
&=1+\frac{4(n+1)^{2}+4n^{2}-x^{2}}{4-z^{2}}z^{2}.
\end{align*}
Now let $\alpha_{n}=4n^{2}(x^{2}-(2n)^{2})$ and 
$\beta_{n+1}=4(n+1)^{2}+4n^{2}-x^{2}=8n^{2}+8n+4-x^{2}$; then 
\[
F^{(1)}(z) 
=\cfrac{\frac{2xt^{2}}{4-z^{2}}z^{2}}{1+\frac{2^{2}-x^{2}}{4-z^{2}}z^{2}
+\cfrac{\frac{\alpha_{1}z^4}{(4-z^2)^{2}}}{1+\frac{\beta_{2}z^2}{4-z^2}
+\cfrac{\frac{\alpha_{2}z^4}{(4-z^2)^{2}}}{1+\frac{\beta_{3}z^2}{4-z^2}+\ddots}}}.
\]
Using \eqref{3.13} with $r_{m}=(4-z^{2})/4$, we finally get
\begin{equation}\label{5.21}
F^{(1)}(z)
=\cfrac{\frac{x}{2}}{1+\frac{2^{2}-1-x^{2}}{4}z^{2}
+\cfrac{\frac{\alpha_{1}}{16}z^{4}}{1+\frac{\beta_{2}-1}{4}z^{2}
+\cfrac{\frac{\alpha_{2}}{16}z^{4}}{1+\frac{\beta_{3}-1}{4}z^{2}+\ddots}}}.
\end{equation}
If we note that 
\[
\frac{x}{2}=E_{1}(\tfrac{1+x}{2}),\quad
\frac{3-x^{2}}{4}=\sigma_{0}^{(1)},\quad
\frac{\beta_{n+1}-1}{4}=\sigma_{n}^{(1)},
\quad\text{and}\quad
\frac{\alpha_{n}}{16}=\tau_{n}^{(1)}
\quad (n\geq 1),
\]
then we see that \eqref{5.21} gives \eqref{5.10a} for $\nu=1$.

To deal with the final case, $\nu=2$, we use \eqref{5.13} in the form
$z^2F^{(2)}(z)=F^{(0)}(z)-1$, and combine it with \eqref{5.19}, namely
\[
F^{(0)}(z) = \cfrac{1}{1+\cfrac{\frac{a_{1}}{4}z^{2}}
{1+\cfrac{\frac{a_{2}}{4}z^{2}}{1+\ddots}}},
\]
where $a_{2n+1}=(2n+1)^{2}-x^{2}$ and $a_{2n}=4n^{2}.$ Now we apply the odd 
canonical contraction \eqref{3.16} with $\alpha_{n}=-a_{n}/4$ and $t=z^{2}$,
obtaining
\begin{equation}\label{5.22}
z^{2}F^{(2)}(z) =\cfrac{\alpha_{1}z^{2}}{1-(\alpha_{1}+\alpha_{2})z^{2}
-\cfrac{\alpha_{2}\alpha_{3}z^{4}}{1-(\alpha_{3}+\alpha_{4})z^{2}
-\cfrac{\alpha_{4}\alpha_{5}z^{4}}{1-(\alpha_{5}+\alpha_{6})z^{2}
-\cfrac{\alpha_{6}\alpha_{7}z^{4}}{\ddots}}}}.
\end{equation}
Finally, it is easy to check that
\[
-\alpha_{1}=E_{2}(\tfrac{1+x}{2}),\quad 
-\alpha_{2n+2}\alpha_{2n+3}=\tau_{n+1}^{(2)},\quad\text{and}\quad 
-\alpha_{2n+1}-\alpha_{2n+2}=\sigma_{n}^{(2)}\quad (n\geq 0).
\]
This, with \eqref{5.22}, gives \eqref{5.10a} for $\nu=2$, and the proof is
complete.
\end{proof}

\section{Shifted sequences}\label{sec:ShiftedSequences}

We saw in the proof of Lemma~\ref{lem:5.3} that the Hankel determinant 
\eqref{5.9} for the case $\nu=2$ was obtained via \eqref{5.12a}, and this also
required some manipulation of the appropriate continued fractions. An 
alternative approach is provided through some results in the literature that 
make it possible to obtain \eqref{5.9} as a consequence of \eqref{5.7}
without using the generating functions \eqref{5.10}. We will
describe this now, along with some further remarks. Since Theorem~\ref{thm:5.1}
and Corollary~\ref{cor:5.2} are proven for the case $\nu=2$, and for the sake
of brevity, we will give only sketches of proofs in what follows.

We begin with the observation that if we set
\begin{equation}\label{6.1}
c_k^{(\nu)}:=E_{2k+\nu}(\tfrac{x+1}{2}),\qquad \nu=0, 1, 2,\ldots,
\end{equation}
as in Theorem~\ref{thm:5.1}, then
\begin{equation}\label{6.2}
H_n(c_k^{(\nu+2)})=H_n(c_{k+1}^{(\nu)}).
\end{equation}
Although the results in Section~5 are valid only for $\nu=0$, 1, and 2, we can
consider \eqref{6.1} and \eqref{6.2} for any integer parameter $\nu\geq 0$. 

We now take a more general approach, based on results in \cite{Ha16} and
\cite{MWY}.
Given a sequence ${\bf a}=(a_0,a_1,\ldots)$, let $P_n(y)$, $n=0,1,\ldots$, 
be the polynomials
orthogonal with respect to ${\bf a}$, and suppose that it satisfies the
recurrence relation $P_0(y)=1$, $P_1(y)=y+s_0$, and
\[
P_{n+1}(y) = (y+s_n)P_n(y) - t_nP_{n-1}(y)\qquad (n\geq 1),
\]
as in \eqref{3.6}. Following \cite{MWY}, we consider the infinite band matrix
\begin{equation}\label{6.4}
J:=\begin{pmatrix}
-s_{0} & 1 & 0 & 0 &\cdots \\
t_{1} & -s_{1} & 1 & 0 & \cdots \\
0 & t_{2} & -s_{2} & 1 & \cdots \\
\vdots & \vdots & \ddots & \ddots & \ddots 
\end{pmatrix}.
\end{equation}
Furthermore, for each $n\geq 0$ let $J_{n}$ be the $(n+1)$th leading principal 
submatrix of $J$ and let $d_{n}:=\det{J_{n}}$, so that $d_0=-s_0$. We also set
$d_{-1}=1$ by convention.

We can now quote the following results.

\begin{proposition}[{\cite[Prop.~1.2]{MWY}}]\label{prop:6.2}
With notation as above, for a given sequence ${\bf a}$ we have
\begin{equation}\label{6.5}
H_{n}(a_{k+1})=H_{n}(a_{k})\cdot d_{n},
\end{equation}
and 
\begin{equation}\label{6.6}
H_{n}(a_{k+2})=H_{n}(a_{k})\cdot\left(\prod_{\ell=1}^{n+1}t_{\ell}\right)
\cdot\sum_{\ell=-1}^{n}\frac{d_{\ell}^{2}}{\prod_{j=1}^{\ell+1}t_{j}}.
\end{equation}
\end{proposition}

\begin{proposition}[{\cite[Eq.~(2.4)]{Ha16}}]\label{prop:6.3}
For a given sequence ${\bf a}$ and $(s_n)$ as defined above, we have
\begin{equation}\label{6.7}
s_{n}=-\frac{1}{H_{n-1}(a_{k+1})}
\left(\frac{H_{n-1}(a_{k})H_{n}(a_{k+1})}{H_{n}(a_{k})}
+\frac{H_{n}(a_{k})H_{n-2}(a_{k+1})}{H_{n-1}(a_{k})}\right).
\end{equation}
\end{proposition}

Set ${\bf c}^{(\nu)}=(c_{0}^{(\nu)},c_{1}^{(\nu)},\ldots)$ and also recall \eqref{6.1} and \eqref{6.2}. 
We now use Propositions~\ref{prop:6.2} and~\ref{prop:6.3} with 
${\bf a}={\bf c}^{(0)}$. With \eqref{6.2} and \eqref{6.5} we get
$H_{n}(c_k^{(2)})$, provided we know $d_n$ (see below). With \eqref{3.10}
we then get $\tau_n^{(2)}$, and \eqref{6.6} and \eqref{6.7} together give 
$\sigma_n^{(2)}$. So 
altogether we would have everything we need to know for the case $\nu=2$ in
Section~5, confirming this part of Theorem~\ref{thm:5.1} and 
Corollary~\ref{cor:5.2}. The identity \eqref{6.6} would also give us
$H_{n}(c_k^{(4)})$.

It remains to determine the factor $d_n$ in \eqref{6.5}. For this 
purpose we use the following lemma.

\begin{lemma}\label{lem:6.4}
For $\nu=0,1,2$, let the sequences $\big(\sigma_n^{(\nu)}\big)_{n\geq 0}$ and
$\big(\tau_n^{(\nu)}\big)_{n\geq 1}$ be defined as in \eqref{5.5}. Furthermore,
let
\[
J^{(\nu)}:=\begin{pmatrix}
-\sigma_{0}^{(\nu)} & 1 & 0 & 0 & \cdots \\
-\tau_{1}^{(\nu)} & -\sigma_{1}^{(\nu)} & 1 & 0 & \cdots \\
0 & -\tau_{2}^{(\nu)} & -\sigma_{2}^{(\nu)} & 1 & \cdots \\
\vdots & \vdots & \ddots & \ddots & \ddots 
\end{pmatrix}
\]
and let $d_{n}^{(\nu)}:=\det{J_{n}^{(\nu)}}$ be the determinant of the
$(n+1)$th leading principal submatrix of $J^{(\nu)}$. Then
\begin{equation}\label{6.9}
d_{n}^{(0)}=\prod_{\ell=0}^{n}\frac{x^{2}-(2\ell+1)^{2}}{4}.
\end{equation}
\end{lemma}

\begin{proof}
Using basic operations with determinants, we obtain the recurrence relation
\begin{equation}\label{6.10}
d_{n+1}^{(\nu)}=-\sigma_{n+1}^{(\nu)}d_{n}^{(\nu)}
+\tau_{n+1}^{(\nu)}d_{n-1}^{(\nu)},
\end{equation}
which is actually independent of the particular values of $\sigma_{n+1}^{(\nu)}$
and $\tau_{n+1}^{(\nu)}$. Using this relation and \eqref{5.5}, the desired 
identity \eqref{6.9} is obtained by induction.
\end{proof}

Combining \eqref{6.9} with \eqref{6.5} and \eqref{6.2}, we have therefore 
shown
\begin{equation}\label{6.11}
H_{n}(c_k^{(2)})=H_{n}(c_k^{(0)})
\prod_{\ell=0}^{n}\frac{x^{2}-(2\ell+1)^{2}}{4};
\end{equation}
it is now easy to see that this is consistent with \eqref{5.7} and \eqref{5.9}.

The results in this section also provide an answer as to why the results in
Section~5 do not extend to $\nu=3$ and beyond. Indeed, in analogy to 
\eqref{6.11} we get from \eqref{6.2} and \eqref{6.5},
\begin{equation}\label{6.12}
H_{n}(c_k^{(3)})=H_{n}(c_k^{(1)})\cdot d_n^{(1)}.
\end{equation}
Using the recurrence \eqref{6.10}, along with the values in \eqref{5.5}, we can
easily compute $d_n^{(1)}$ for $n=0,1,\ldots$; the first few are shown in
Table~3.

\bigskip
\begin{center}
{\renewcommand{\arraystretch}{1.2}
\begin{tabular}{|c||l|}
\hline
$n$ & $d_n^{(1)}$ \\
\hline
0 & $\frac{1}{4}(x^{2}-3)$ \\
1 & $\frac{1}{16}(x^{4}-18x^{2}+41)$ \\
2 & $\frac{1}{64}(x^{6}-53x^{4}+655x^{2}-1323)$ \\
3 & $\frac{1}{256}(x^{8}-116x^{6}+3958x^{4}-41364x^{2}+77841)$ \\
\hline
\end{tabular}}

\medskip
{\bf Table~3}: $d_n^{(1)}$ for $0\leq n\leq 3$.
\end{center}
\bigskip

The identity \eqref{6.12}, combined with \eqref{5.8}, shows that 
$H_{n}(c_k^{(3)})$ still has many linear factors, while also having 
apparently irreducible factors of increasing degrees. The same will hold for
$H_{n}(c_k^{(\nu)})$ for any integer shift $\nu$.

Finally, analogous results are true for the Bernoulli polynomial case of
Section~4. Indeed, by iterating the identities in Propositions~\ref{prop:6.2}
and~\ref{prop:6.3}, especially \eqref{6.5}, one could show that for any
integers $n\geq 1$ and $\nu\geq 1$, the Hankel determinant
$H_n\big(B_{2k+2\nu+1}(\frac{x+1}{2})\big)$ retains the same linear factors
present in $H_n\big(B_{2k+1}(\frac{x+1}{2})\big)$, as given in 
Theorem~\ref{thm:1.1}.

We close by mentioning that, for some special values of shifted Bernoulli 
polynomial sequences, their Hankel determinants have been considered in 
Sections~7 and~8 of \cite{FK}.

\end{document}